\documentclass[a4paper,12pt,reqno]{article}
\usepackage[T1]{fontenc}
\usepackage{authblk}
\usepackage{epsfig,graphicx,subfigure}
\usepackage{amsthm,amsmath,amssymb,amsfonts}
\usepackage{color,xcolor}
\usepackage[all]{xy}
\usepackage[top=38mm, bottom=38mm, left=40mm, right=35mm]{geometry}
\usepackage{hyperref}

\newtheorem{theorem}{Theorem}[section]

\newtheorem{proposition}[theorem]{Proposition}
\newtheorem{corollary}[theorem]{Corollary}
\theoremstyle{definition}

\newtheorem{remark}[theorem]{Remark}
\numberwithin{equation}{section}

\begin{document}
\setcounter{page}{1}
\title{\vspace{-1.5cm}
\vspace{.5cm}
\vspace{.7cm}
{\large{\bf Jordan left $\alpha$-centralizer on certain algebras} }}
\date{}
\author{{\small \vspace{-2mm} M. Eisaei$^1$, M. J. Mehdipour $^2$\footnote{Corresponding author}, Gh. R. Moghimi}$^1$}
\affil{\small{\vspace{-4mm} $^1$ Department of Mathematics, Payam Noor University, Shiraz, Iran. }}
\affil{\small{\vspace{-4mm} $^2$ Department of Mathematics, Shiraz University of Technology, Shiraz, Iran. }}
\affil{\small{\vspace{-4mm} mojdehessaei59@student.pnu.ac.ir }}
\affil{\small{\vspace{-4mm} mehdipour@sutech.ac.ir}}
\affil{\small{\vspace{-4mm} moghimimath@pnu.ac.ir}}
\maketitle
\hrule
\begin{abstract}
\noindent
In this paper, we investigate Jordan left $\alpha$-centralizer on algebras. We show that every Jordan left $\alpha$-centralizer on an algebra with a right identity is a left $\alpha$-centralizer. We also investigate this result for Banach algebras with a bounded approximate identity. Finally, we study Jordan left $\alpha$-centralizer on group algebra $L^1(G)$.
\end{abstract}

\noindent \textbf{Keywords}: Jordan left $\alpha$-centralizer, algebra, compact left centralizer, group algebra.\\
{\textbf{2020 MSC}}: 47B48, 47B47, 43A15.
\\
\hrule
\vspace{0.5 cm}
\baselineskip=.8cm
\section{Introduction}
Let $G$ be a locally compact group with the group algebra $L^1(G)$ and the measurable algebra $M(G)$. 
It is well-known that $M(G)$ is the dual of $C_0(G)$, the space of all complex-valued continuous functions on $G$ vanish at infinity. It is also known that $M(G)$ is a unital Banach algebra with the convolution product. Let us remark that if $\mu, \nu \in M(G)$, then the convolution product $\mu$ and $\nu$ is the following:
$$\mu * \nu (f) = \int_G \int_G f(xy) d\mu(x) d \nu(y)$$
for all $f \in C_0(G)$. Note that $L^1(G)$ is a closed ideal of $M(G)$ and so it is a Banach algebra; for study of these Banach algebras see [6].

Let $A$ be an algebra and $\alpha$ be a homomorphism on $A$. Let us recall that 
an additive function $T: A\rightarrow A$ 
is called a \textit{ left $\alpha$-centralizer} if 
$$T(xy)= T(x) \alpha(y)$$
for all $x, y \in A$. Also, $T$ is called \textit{Jordan left $\alpha$-centralizer} if 
$$T(x^2)= T(x) \alpha(x)$$
for all $x\in A$. In the case where, $\alpha$ is the identity map, $T$ is called \emph{left centralizer} and \emph{Jordan left centralizer}, respectively. 

Zalar [16] studied Jordan left centralizers on semiprime rings and proved that every Jordan left centralizer on a semiprime ring $R$ with $\hbox{char}(R)\neq 2$ is a left centralizer. Wendel [15] investigated left centralizers on group algebras and determined left centralizers on $L^1(G)$. Compact left centralizers on $L^1(G)$ have been studied by Sakai [14] and Akemann [1]. They showed that the existence of a non-zero compact left centralizer on $L^1(G)$ is equivalent to the compactness of $G$; see [4, 5, 8, 11, 12]  for study of compact left centralizers on other group algebras; see also [3].

In this paper, we investigate Jordan left $\alpha$-centralizer on algebras and some group algebras. In Section 2, we show that if $A$ is an algebra with a right identity, then every Jordan left $\alpha$-centralizer on $A$ is a left $\alpha$-centralizer. We also prove that if $\alpha$ is continuous, then this result holds for bounded Jordan left $\alpha$-centralizers on Banach algebras with a bounded approximate identity. In Section 3, we let $\alpha$ be a continuous homomorphism on $L^1(G)$ and characterize bounded Jordan left $\alpha$-centralizer on $L^1(G)$. We also investigate the existence of a non-zero compact Jordan left $\alpha$-centralizer on $L^1(G)$.
\section{Jordan left $\alpha$-centralizer on algebras} 

Throughout this section, $\alpha$ is a homomorphism on an algebra $A$. The main result of this paper is the following.
\begin{theorem} \label{m1} Let $A$ be an algebra with a right identity $u$. Then every Jordan left $\alpha$-centralizer on $A$ is a 
left $\alpha$-centralizer. 
\end{theorem}
\begin{proof}
Let $T:A \rightarrow A$ be a Jordan left $\alpha$-centralizer. 
Then 
$T(x^2)=T(x) \alpha (x)$ 
for all $x \in A$. 
By the linearity of this equation we obtain 
\begin{equation} \label{e1} 
T(xy+yx)= T(x) \alpha(x) + T(y) \alpha(x)
\end{equation}
for all $x,y \in A$. Set $y=u$ in (2.1). Then 
\begin{equation} \label{e2} 
T(x+ u x)= T(x) \alpha(u) + T(u) \alpha(x).
\end{equation}
Putting $x=u$ in (2.2), we have 
\begin{equation} \label{e3} 
T(u)= T(u) \alpha (u). 
\end{equation}
If we replace $x$ by $ux$ in (2.2), then 
\begin{equation} \label{e4} 
2 T(ux)= T(ux) \alpha (u)+ T(u) \alpha(x). 
\end{equation}
This together with (2.2) implies that
\begin{eqnarray} \label{e5} 
2T(x)=2 T(x) \alpha(u) + T(u) \alpha(x) - T(ux) \alpha(u). 
\end{eqnarray}
Take $y=xy+yx$ in (2.1). Then 
\begin{equation} \label{e6} 
T(xyx)= T(x) \alpha(y) \alpha(x). 
\end{equation}
It follows that 
\begin{eqnarray} \label{e7}
T(ux)= T(u xu ) &=&T(u) \alpha(x) \alpha(x) \nonumber \\
&=& T(u) \alpha (xu) \\
&=& T(u) \alpha(x). \nonumber
\end{eqnarray}
Replacing $x$ by $ux$ in (2.5) and using (2.3) and (2.7), we have 
\begin{eqnarray*}
2T(ux) &=& 2 T(ux) \alpha(u) + T(u) \alpha(x) - T(ux) \alpha(u) \\
&=& T(ux) \alpha(u) + T(u) \alpha(x) \\
&=& T(ux) \alpha(u) + T(ux). 
\end{eqnarray*}
This shows that
\begin{equation} \label{e8} 
T(ux)= T(ux) \alpha(u).
\end{equation}
From (2.7) and (2.8), we infer that 
\begin{equation*} 
T(u) \alpha(x)= T(ux) \alpha(u).
\end{equation*}
On the other hand, 
by (2.2) and (2.7) 
\begin{equation*} 
T(x) + T(u) \alpha(x) = T(x) \alpha(u)+ T(u) \alpha(x). 
\end{equation*}
So 
\begin{equation*} 
T(x)= T(x) \alpha(x).
\end{equation*}
Let $r:= x-ux$. Then 
\begin{equation*} 
T(r)= T(r) \alpha(u).
\end{equation*}
Setting $x= u+r$ in (2.6), we have
\begin{eqnarray} \label{e9}
T(uy+ry) &=& T(u+r) y (u+r) \nonumber \\
&=& T(u+r) \alpha(y) \\
&=& T(u) \alpha(y) + T(r) \alpha(y)\nonumber
\end{eqnarray}
for all $y \in A$. By (2.7) 
\begin{equation*} 
T(uy+ry) =T(uy) + T(ry)= T(u) \alpha(y)+ T(ry). 
\end{equation*} 
From this and (2.9) we see that 
\begin{equation*} 
T(ry)=T(r) \alpha(y). 
\end{equation*} 
Hence for every $x,y \in A$, we have
\begin{eqnarray*} 
T(xy) &=& T(xy) - T(u) \alpha (xy) + T(u) \alpha (xy) \\
&=& T(xy) - T(uxy) + T(u) \alpha(xy) \\
&=& T((x-ux)y) + T(u) \alpha(xy) \\
&=& T(x-ux) \alpha(y) + T(u) \alpha(xy) \\
&=& T(x) \alpha(y) - T(ux) \alpha(y)+ T(u) \alpha(xy) \\
&=& T(x) \alpha (y). 
\end{eqnarray*} 
Therefore, $T$ is a left $\alpha$-centralizer on $A$. 
\end{proof}
\begin{remark} \label{m10}
(i) Similar to the proof of Theorem 2.1, one can prove that Theorem 2.1 holds for a ring $R$ with $\text{char}(R) \not =2$. \\
(ii) If $A$ is a unital algebra with the identity $e$, and $T: A \rightarrow A$ is a Jordan left $\alpha$-centralizer, then 
\begin{equation*} 
T(x)= T(ex) = T(e) \alpha(x)
\end{equation*}
for all $x \in A$. \\
(iii) There are some well-known group algebras which have a right identity or an identity. For example, the Banach algebras $L_0^\infty(G)^*$ defined as in [7] and $L^1(G)^{**}$ have a right identity; for more study on $L_0^\infty(G)^*$ see [11-13]. Also, the Banach algebras $(M(G)_0^*)^*$ defined as in [9, 10], $M(G)$ and $\hbox{LUC}(G)^*$ have an identity, where $\hbox{LUC}(G)$ is the space of all left uniformly continuous functions on $G$; for more details see [2]. Hence every Jordan left $\alpha$-centralizer on these Banach algebras is a left $\alpha$-centralizer.
\end{remark}
\begin{corollary}
Let $A$ be a unital normed algebra and $\alpha$ be a continuous homomorphism on $A$. Then every Jordan left $\alpha$-centralizer on $A$ is continuous. 
\end{corollary}
\begin{proof}
Let $A$ be a unital normed algebra with the identity $e$. Let $T$ be a Jordan left $\alpha$-centralizer on $A$. Then by Remark 2.2 (ii), 
$$T(x)= T(e) \alpha(x)$$
for all $x \in A$. So if $\alpha$ is continuous, then 
\begin{eqnarray*}
\| T(x) \| = \| T(e) \alpha(x) \| \leq \|T(e)\| \| \alpha \| \| x\|. 
\end{eqnarray*}
Therefore, $T$ is continuous. 
\end{proof}
Let $A $ be a Banach algebra.
In the next result, we equip the second conjugate $A $ with the first Arens product. 
Let us recall that if $m,n \in A ^{**}$, then the first Arens product $m$ and $n$ is defined by $\langle mn, f \rangle= \langle m, nf \rangle$, where $\langle nf, a \rangle= \langle n, f a \rangle$ and $\langle fa, b \rangle = \langle f, ab \rangle$ for all $f \in A ^*$ and $a,b \in A $. 

\begin{theorem} \label{m7}
Let $A $ be Banach algebra with a bounded approximate identity and $\alpha$ be a continuous homomorphism on $A$. 
Then every bounded Jordan left $\alpha$-centralizer on $A $ is a left $\alpha$-centralizer. 
\end{theorem}
\begin{proof} 
Let $T: A \rightarrow A $ be a bounded Jordan left $\alpha$-centralizer. 
It is easy to see that 
$T^*(f) \, a = \alpha^* \, ( f \, T(a))$ and so 
$n \, T^* (f)= T^* ( \alpha ^{**} (n) f)$ 
for all $n \in A ^{**}$, $ f \in A ^*$ and $a \in A $. 
Hence 
$T^{**} (mn) = T^{**}(m) \alpha^{**} (n)$ for all $m,n \in A ^{**}$. 
That is, $T^{**}$ 
is a Jordan left $\alpha^{**}$-centralizer on $A ^{**}$. 
Since $A $ has a bounded approximate identity, $A ^{**}$ has a right identity. 
In view of Theorem 2.1, $T^{**}$ is a left $\alpha^{**}$-centralizer on $A ^{**}$. 
But, 
$T^{**}|_{A}= T$
and $\alpha^{**}|_{A}=\alpha$. 
Therefore, 
$T$ is a left $\alpha$-centralizer on $A $. 
\end{proof} 
Since every $C^*$-algebra has a bounded approximate identity, we have the following result. 
\begin{corollary} Let $\alpha$ be a continuous homomorphism on $C^*$-algebra $A$. Then every bounded Jordan left $\alpha$-centralizer on $A$, is a left $\alpha$-centralizer. 
\end{corollary}
We finish this section with the following result.
\begin{proposition} \label{m4} 
Let $A$ be a unital Banach algebra and $\alpha$ be a continuous epimorphism on $A$. Then there is a non-zero compact Jordan left centralizer on $A$ if and only if there is a non-zero compact Jordan left $\alpha$-centralizer on $A$. 
\end{proposition} 
\begin{proof} 
Let $T: A \rightarrow A$ be a non-zero compact Jordan left centralizer. Then there exists $b \in A$ such that $T(a)=ba$ for all $a \in A$. 
Define $\tilde{T}: A \rightarrow A$ by $\tilde{T}= T \circ \alpha $. 
Clearly, $\tilde{T}$ is a Jordan left $\alpha$-centralizer on $A$. 
Since $T$ is compact, $\tilde{T}$ is compact. 
Now, let $a_0 \in A$ and $T(a_0) \not =0$. Choose $x_0 \in A$ with 
$\alpha (x_0)= a_0$. Then 
\begin{equation*} 
\tilde{T}(x_0)= T \circ \alpha (x_0) = T(a_0) \not = 0. 
\end{equation*}
Hence $\tilde{T}$ is a non-zero compact Jordan left $\alpha$-centralizer on $A$. 

Conversely, let $\hat{T}: A \rightarrow A$ be a non-zero compact Jordan left $\alpha$-centralizer on $A$. 
Hence there exists $b \in A$ such that $\hat{T}(a) = b \alpha(a)$ for all $a \in A$;
in fact, $b= \hat{T}(1_A)$. 
Let $T: A \rightarrow A$ defined by $T(a)=ba$. 
Let $(a_n)$ be a sequence in $A$.
Since $\alpha$ is onto, for every $n\in {\Bbb N}$ there exists $x_n\in A$ such that $\alpha(x_n)= a_n$. 
The compactness of $\hat{T}$ 
shows that the sequence $(x_n)$ has a 
subsequence $(x_{n_k})$ 
such that $(\hat{T}(x_{n_k}))$ is convergent. 
So $(T(a_{n_k}))$ is convergent. 
That is, $T$ is compact. 
\end{proof}
\section{Left $\alpha$-centralizer on group algebra $L^1(G)$} 

Before, we give the first result of this section, let us remark that since $L^1(G)$ has a bounded approximate identity, every Jordan left $\alpha$-centralizer on $L^1(G)$ is a left $\alpha$-centralizer, where $\alpha$ is a continuous homomorphism on $L^1(G)$. The next result is an analogue of Wendel's theorem [15].
\begin{theorem} 
Let $G$ be a locally compact group and $\alpha$ be a continuous homomorphism on $L^1(G)$. 
If $T$ is a bounded Jordan left $\alpha$-centralizer on $L^1(G)$, then there exists $\mu \in M(G)$ such that $T(\varphi) = \mu * \alpha(\varphi)$ for all $ \varphi \in L^1(G)$. 
\end{theorem}
\begin{proof} 
Let $(e_{\alpha})$ be a bounded approximate identity for $L^1(G)$. Since $T$ is bounded, $(T(e_{\alpha}))_{\alpha}$ can be regarded as a bounded net in $M(G)= C_0(G)^*$. 
Hence $(T(e_{\alpha}))_{\alpha}$ has a weak$^*$-cluster point, say $\mu \in M(G)$. 
Without loss of generality, we may assume that $T(e_{\alpha}) \rightarrow \mu$ in the weak$^*$ topology of $M(G)$. 
So if $\varphi \in L^1(G)$, then 
\begin{equation*} 
T(e_{\alpha} * \varphi ) = T(e_{\alpha})* \alpha (\varphi) \rightarrow \mu * \alpha (\varphi)
\end{equation*}
in the weak$^*$ topology of $M(G)$. 
On the other hand, $T(e_{\alpha} *\varphi) \rightarrow T(\varphi)$ in the norm topology. Therefore, $T(\varphi) = \mu * \alpha( \varphi)$. 
\end{proof}
In the next result, we characterize compact Jordan left $\alpha$-centralizers on $L^1(G)$, when $\alpha$ is a continuous epimorphism on $L^1(G)$; see Akemann [1].
\begin{corollary}
Let $G$ be a locally compact group and $\alpha$ be a continuous epimorphism on $L^1(G)$. If $T$ is a compact Jordan left $\alpha$-centralizer on $L^1(G)$, then there exists $\psi \in L^1(G)$ such that $T(\varphi)= \psi * \alpha (\varphi)$ for all $\varphi \in L^1(G)$.
\end{corollary}
\begin{proof}
Let $\alpha$ be a continuous epimorphism on $L^1(G)$ and $T$ be a compact Jordan left $\alpha$-centralizer on $L^1(G)$. 
Then there exists $\mu\in M(G)$ such that $T(\varphi) = \mu * \alpha(\varphi)$ for all $ \varphi \in L^1(G)$. So the left centralizer $\phi\mapsto \mu\ast\phi$ is compact on $L^1(G)$. Hence $\mu \in L^1(G)$; see [1,14]. 
\end{proof}
\begin{proposition} 
Let $G$ be a locally compact group and $\alpha$ be a continuous epimorphism on $L^1(G)$. Then the following assertions are equivalent. 

\emph{(a)} There is a non-zero compact Jordan left $\alpha$-centralizer on $L^1(G)$. 

\emph{(b)} Every Jordan left $\alpha$-centralizer on $L^1(G)$ is compact. 

\emph{(c)} $G$ is compact. 
\end{proposition}
\begin{proof} 
It is proved that $G$ is compact if and only if there is a non-zero compact left centralizer on $L^1(G)$; or equivalently, every left centralizer on $L^1(G)$ is compacts, see [1,14]. These facts together with Proposition 2.6 prove the result. 
\end{proof}

\end{document}